\documentclass[11pt]{article}

\usepackage{amsmath,amsfonts,amsthm,amscd,amssymb,graphicx,mathtools}
\usepackage{cite}
\usepackage{bbm}
\usepackage{color}
\usepackage[english=american]{csquotes}
\usepackage[colorlinks=true, pdfstartview=FitV, linkcolor=black, citecolor=black, urlcolor=black]{hyperref}
\usepackage{calc}
\usepackage{mathptmx}
\usepackage{t1enc}
\usepackage{stackrel}
\usepackage{dsfont}
\usepackage{enumerate}
\usepackage{hyperref}

\newtheorem{theo}{Theorem}[section]
\newtheorem{lem}[theo]{Lemma}
\newtheorem{prop}[theo]{Proposition}

\newtheorem{cor}[theo]{Corollary}

\newtheorem{rem}[theo]{Remark}

\DeclareMathOperator{\dist}{dist}

\DeclareMathOperator{\diverg}{div}
\DeclareMathOperator{\curl}{curl}

\newcommand{\N}{\mathbb{N}}
\newcommand{\R}{\mathbb{R}}

\newcommand{\Z}{\mathbb{Z}}

\newcommand{\eps}{\epsilon}

\newcommand{\dx}{\,\mathrm{d}x}
\newcommand{\dt}{\,\mathrm{d}t}

\newcommand{\dy}{\,\mathrm{d}y}
\newcommand{\dz}{\,\mathrm{d}z}
\newcommand{\dd}{\,\mathrm{d}}

\newcommand{\image}{\mathrm{im}\;}

\renewcommand{\eps}{\varepsilon}
\renewcommand{\epsilon}{\varepsilon}

 \def\Xint#1{\mathchoice
 	{\XXint\displaystyle\textstyle{#1}}%
 	{\XXint\textstyle\scriptstyle{#1}}%
 	{\XXint\scriptstyle\scriptscriptstyle{#1}}%
 	{\XXint\scriptscriptstyle\scriptscriptstyle{#1}}%
 	\!\int}
 \def\XXint#1#2#3{{\setbox0=\hbox{$#1{#2#3}{\int}$ }
 		\vcenter{\hbox{$#2#3$ }}\kern-.58\wd0}}
 
 \def\dashint{\Xint-}

\numberwithin{equation}{section}
\DeclareMathAlphabet{\mathcal}{OMS}{cmsy}{m}{n}

\begin{document}
\allowdisplaybreaks
	\title{M\"uller-Zhang truncation for general linear constraints with first or second order potential}
	
	\author{Dennis Gallenm\"uller\footnotemark[1]}
	
	\date{}
	
	\maketitle
	
	\begin{abstract}
		Let $\mathcal{B}$ be a homogeneous differential operator of order $l=1$ or $l=2$. We show that a sequence of functions of the form $(\mathcal{B}u_j)_j$ converging in the $L^1$-sense to a compact, convex set $K$ can be modified into a sequence converging uniformly to this set provided that the derivatives of order $l$ are uniformly bounded. We prove versions of our result on the whole space, an open domain, and for $K$ varying uniformly continuously on an open, bounded domain. This is a conditional generalization of a theorem proved by S.~M\"uller for sequences of gradients. Moreover, a potential of order two for the linearized isentropic Euler system is constructed.

		
	\end{abstract}
	
	\renewcommand{\thefootnote}{\fnsymbol{footnote}}
	
	\footnotetext[1]{Institut f\"ur Angewandte Analysis, Universit\"at Ulm, Helmholtzstra\ss e 18, 89081 Ulm, Germany. Email: dennis.gallenmueller@uni-ulm.de}

	\section{Introduction}
	In the calculus of variations, particularly in the context of quasiconvexity and gradient Young measures, understanding sequences of gradients is crucial. If we are given such a sequence of gradients with a certain limit behavior, one might ask if there exists another sequence of this form with better regularity or convergence properties but exhibiting the same limit behavior. In this context, K.~Zhang \cite{Z92} showed that a sequence of weakly differentiable functions $(u_j)_j$ on $\R^d$ which converges to a ball $B_R(0)$ for $R>0$ in the sense that
	\begin{align}
		\int\limits_{\R^d}^{}\dist(Du_j,B_R(0))\dx\rightarrow 0\label{eq:zhang}
	\end{align}
	can be modified into a sequence $(w_j)_j$ of uniformly Lipschitz functions. This modification takes place on sets whose measure decreases with increasing $j$ such that one obtains $|\{u_j\neq w_j \}|\rightarrow 0$. Zhang's result has been developed further by S.~M\"uller, see \cite{M99}. In his paper M\"uller presents a method to regularize a sequence of $W^{1,1}$-functions satisfying (\ref{eq:zhang}) where the ball $B_R(0)$ is replaced by any compact, convex set $K$ in such a way that the convergence of the new sequence $(w_j)_j$ is uniform, i.e.
	\begin{align*}
		\|\dist(Dw_j,K)\|_{L^{\infty}}\rightarrow 0,
	\end{align*}
	and also the sets of modification have measure tending to zero, hence $|\{u_j\neq w_j \}|\rightarrow 0$.\\
	Our aim in the present work is to generalize M\"uller's method to $L^1$-convergent sequences of functions of the form $(\mathcal{B}u_j)_j$, where $\mathcal{B}$ is a linear homogeneous differential operator of order $l=1$ or $l=2$. For that, we will have to additionally assume that $\|D^lu_j\|_{L^{\infty}}$ is uniformly bounded in order to obtain a uniformly convergent sequence, cf. Theorem \ref{theo:muellertheo2} below.\\
	The motivation for the present work was a better understanding of limit processes under general linear constraints. Especially improving the properties of generating sequences for Young measures was an important aspect, cf. Corollary \ref{cor:muellercorollary3}. There are prominent examples of linear constraints with potentials of first or second order to which our results apply. For example the symmetric gradient is a first order potential for a second order linear constraint. Furthermore, the linearized compressible Euler equations come with a second order potential operator. The latter two examples will be discussed in Section \ref{sect:twoexamples}.\\
	For proving our result we roughly follow the strategy of M\"uller in \cite{M99}, but we need to introduce a different method of regularization which is crucial for our generalization. The outline of M\"uller's proof, and also of our proof, goes as follows:\\
	First obtain an $L^1$-estimate for a certain regularization scheme on balls. Then determine and control the set of balls on which the regularization has to be applied. Finally, one iterates this procedure for each function $u_j$ to arrive at the uniformly convergent sequence $(w_j)_j$.\\
	In attempting such a generalization there arise some difficulties. One consists of the fact that for general homogeneous operators $\mathcal{B}$ of order $l$ it is not true that $|D^lu|\leq C|\mathcal{B}u|$ holds pointwise a.e. for a fixed constant $C$. To circumvent this problem in the proof we additionally need to assume that the derivatives of order $l$ are uniformly bounded. This assumption is of technical nature. However, in Remark \ref{rem:refereequeryone} below we give an example of a situation where this is naturally given. Without this bound estimating the occurring highest order derivatives is not possible.\\
	Furthermore, in the case $l=2$ we will face a more subtle problem: One has to estimate terms of the form $|u_j(x)-u_j(y)|$ which is not possible with a mere bound on $\|D^2u_j\|_{L^{\infty}}$. We therefore need to make use of a different type of regularization on balls, see Lemma \ref{lemm:muellerlemma5}, which should be compared with Lemma 5 in \cite{M99}. Instead of mollifying with a fixed radius and achieving the transition to the original function by multiplication with a smooth cut-off function $\varphi$ and its counterpart $(1-\varphi)$, we regularize by mollifying with varying radius. More precisely, the radius of the mollification is given by a smooth function decreasing to zero with increasing distance to the center of the ball. By Lebesgue's differentiation theorem the so defined regularized function coincides with the original function as soon as the radius of mollification tends to zero. The benefit of using this method is that differentiating the new function yields no zeroth order derivatives of $u_j$ and the terms involving the first order derivatives can be estimated using the fundamental theorem of calculus and the uniform bounds on $\|D^2u_j\|_{L^{\infty}}$, cf. the proof of Lemma \ref{lemm:muellerlemma5}. Note that for orders $l>2$ one would have to come up with yet another regularization scheme since then with our method the lower order derivatives cannot be controlled anymore.\\
	After proving our main result on $\R^d$ we also present a local version thereof in Theorem \ref{theo:mueller4} and we moreover show that certain convex integral bounds are preserved. These local results lead to interesting applications: In Corollary \ref{cor:varyingk} we show that $K$ can be replaced by a family of compact, convex sets $(K_x)_x$ depending uniformly continuously on $x$. Moreover, our local result allows for a reformulation in the language of Young measures, cf. Corollary \ref{cor:muellercorollary3}.\\
	The latter Young measure formulation of our result is an important intermediate goal of an ongoing joint project with E.~Wiedemann. We try to characterize measure-valued solutions to the isentropic Euler system that are generated by weak solutions. For the incompressible Euler equations every measure-valued solution is generated in this way, cf. \cite{SW12}. In contrast to that, some recent work shows that there exist measure-valued solutions to the isentropic Euler system that are neither generated by weak solutions nor by a vanishing viscosity limit, cf. \cite{CFKW17} and \cite{GW20}. We may not apply our results to the compressible Euler system directly but to its linearization (\ref{eq:CEsubsolution}). The latter possesses a potential of order two, which was pointed out to me by E.~Wiedemann in dimension $3+1$. At the end of the present paper we give a construction of this potential for dimensions $d+1\geq 2$.\\
	Besides the linearized compressible Euler equations other partial differential operators also have low order potentials. One already mentioned prominent example is the corresponding linear constraint to the symmetric gradient as potential. However, a further generalization of our results to higher order operators is desirable as the work of B.~Rai\c{t}\u{a} \cite{R18} shows that any constant-rank homogeneous differential operator has a potential which is in general of high order. This generalization would have consequences for the theory of general $\mathcal{A}$-free Young measures and $\mathcal{A}$-quasiconvex functions, developed e.g. in \cite{A19}, \cite{KR19}, and \cite{FM99}.\\
	The organization of this paper is as follows. In Section \ref{sect:presentationofmainresults} we formulate our main results and prove some direct consequences. The main proofs are presented in Section \ref{sect:proofsinrd}. At the end of the paper, we present two examples of linear constraints with potentials of order one or two. For that we derive in Section \ref{sect:euler} a potential of order two for the linearized isentropic Euler system.	
\section{Presentation of the main results}\label{sect:presentationofmainresults}
	Let $\mathcal{B}=\sum\limits_{|\alpha|=l}^{}B^{\alpha}\partial_{\alpha}$ denote a homogeneous differential operator of order $l\in \{1,2\}$ with $B^{\alpha}\in \R^{k\times m}$ constant. We will use the notation in \cite{M99}, which we will briefly revise for convenience:\\
	In what follows $K\subset \R^k$ denotes a compact, convex set. We define $|K|_{\infty}:=\max\{|A|\,:\,A\in K \}$. The sublevel sets
	\begin{align*}
		K_{\gamma}:=\{B\subset \R^k\,:\,\dist(B,K)\leq \gamma \}
	\end{align*}
	are again compact and convex. Moreover, one has
	\begin{align*}
		\dist(B,K_{\gamma})\leq (\dist(B,K)-\gamma)^+,
	\end{align*}
	where $b^+=\max(b,0)$ for $b\in\R$.\\
	Our main result on the whole space $\R^d$, which corresponds to Theorem 2 in \cite{M99}, will be a direct consequence of Proposition \ref{prop:muellerprop7undcor8} below.
	\begin{theo}\label{theo:muellertheo2}
		Let $K$ be a compact, convex set in $\R^k$ and $l=1$ or $l=2$. Further, let $(u_j)\in W^{l,\infty}_{\operatorname{loc}}(\R^d,\R^m)$ be a sequence of functions such that $\|D^lu_j\|_{L^{\infty}(\R^d)}\leq M|K|_{\infty}$ and
		\begin{align*}
			\int\limits_{\R^d}^{}\dist(\mathcal{B}u_j,K)\dx\rightarrow 0.
		\end{align*}
		Then there exist $(w_j)\in W^{l,\infty}(\R^d,\R^m)$ such that
		\begin{align*}
			\|\dist(\mathcal{B}w_j,K)\|_{L^{\infty}(\R^d)}&\rightarrow 0,\\
			|\{u_j\neq w_j \}|&\rightarrow 0.
		\end{align*}
	\end{theo}
	\begin{theo}\label{theo:mueller4}
		Let $K$ be a compact, convex set in $\R^k$, let $\Omega\subset \R^d$ be open and let $l=1$ or $l=2$. Further, let $(u_j)\in W^{l,\infty}_{\operatorname{loc}}(\Omega,\R^m)$ be a sequence of functions satisfying $\|D^lu_j\|_{L^{\infty}(\Omega)}\leq M|K|_{\infty}$ and
		\begin{align*}
			u_j&\rightarrow u_0\text{ in }L^1_{\operatorname{loc}}(\Omega,\R^m),\\
			\dist(\mathcal{B}u_j,K)&\rightarrow 0\text{ in }L^1(\Omega).
		\end{align*}
		Then there exist functions $w_j\in W^{l,\infty}_{\operatorname{loc}}(\Omega,\R^m)$ and an increasing sequence of open sets $(U_j)$, which are compactly contained in $\Omega$, such that
		\begin{align*}
			w_j&=u_0\text{ on }\Omega\backslash U_j,\\
			|\{u_j\neq w_j \}\cap U_j|&\rightarrow 0,\\
			\|\dist(\mathcal{B}w_j,K)\|_{L^{\infty}(\Omega)}&\rightarrow 0.
		\end{align*}
		Moreover, if $|\Omega|<\infty$ then $|\{u_j\neq w_j \}|\rightarrow 0$.
	\end{theo}
	One can ask if there are quantities which are conserved under the truncation method that leads to the previous theorems. Indeed, we can show that certain convex integral functionals are preserved in the limit. In particular this holds for the $p$-norm. We only state the result in the local formulation, but this also holds for $\Omega=\R^d$.
	\begin{cor}\label{cor:convexintegralbound}
		Let $K$ be a compact, convex set in $\R^k$, let $\Omega\subset \R^d$ be open, bounded and let $l=1$ or $l=2$. Further, let $F\colon \R^k\to \R$ be a non-negative, homogeneous, convex function and let $(u_j)\in W^{l,\infty}_{\operatorname{loc}}(\Omega,\R^m)$ be a sequence of functions satisfying $\|D^lu_j\|_{L^{\infty}(\Omega)}\leq M|K|_{\infty}$ and
		\begin{align*}
			u_j&\rightarrow u_0\text{ in }L^1_{\operatorname{loc}}(\Omega,\R^m),\\
			\dist(\mathcal{B}u_j,K)&\rightarrow 0\text{ in }L^1(\Omega).
		\end{align*}
		Then there exist functions $w_j\in W^{l,\infty}_{\operatorname{loc}}(\Omega,\R^m)$ and an increasing sequence of open sets $(U_j)$, which are compactly contained in $\Omega$, such that
		\begin{align*}
			w_j&=u_0\text{ on }\Omega\backslash U_j,\\
			|\{u_j\neq w_j \}|&\rightarrow 0,\\
			\|\dist(\mathcal{B}w_j,K)\|_{L^{\infty}(\Omega)}&\rightarrow 0,\\
			\underset{j\rightarrow\infty}{\limsup}\int\limits_{\Omega}^{}F(\mathcal{B}w_j)\dx&\leq \underset{j\rightarrow\infty}{\limsup}\int\limits_{\Omega}^{}F(\mathcal{B}u_j)\dx<\infty.
		\end{align*}
	\end{cor}
	From the previous local formulation of our main result we obtain a generalization of Corollary 3 in \cite{M99} treating gradient Young measures. For the relevant definitions consult e.g.~\cite{FM99}.
	\begin{cor}\label{cor:muellercorollary3}
		Let $K$ be a compact, convex set in $\R^k$, let $\Omega\subset \R^d$ be open, bounded and let $l=1$ or $l=2$. Further, let $F\colon \R^k\to \R$ be a non-negative, homogeneous, convex function. Suppose that $(u_j)\in W^{l,\infty}_{\operatorname{loc}}(\Omega,\R^m)$ is a sequence of functions satisfying
		\begin{align*}
			\|D^lu_j\|_{L^{\infty}(\Omega)}&\leq M|K|_{\infty},\\
			u_j&\rightharpoonup u_0\text{ in }W^{l,p}_{}(\Omega,\R^m)
		\end{align*}
		for some $1\leq p\leq \infty$, where in the case $p=\infty$ the weak convergence is replaced by weak-*-convergence. Further, assume that $(\mathcal{B}u_j)$ generates a Young measure $\nu=(\nu_x)_{x\in\Omega}$ with
		\begin{align*}
			\operatorname{supp}\nu_x\subset K\text{ for a.e. }x\in\Omega.
		\end{align*}
		Then there exists a sequence of functions $(g_j)$ such that $(\mathcal{B}g_j)$ generates the Young measure $\nu$ and
		\begin{align*}
			\|\dist(\mathcal{B}g_j,K)\|_{L^{\infty}(\Omega)}&\rightarrow 0,\\
			\underset{j\rightarrow\infty}{\limsup}\int\limits_{\Omega}^{}F(\mathcal{B}g_j)\dx&\leq \underset{j\rightarrow\infty}{\limsup}\int\limits_{\Omega}^{}F(\mathcal{B}u_j)\dx<\infty.
		\end{align*}
	\end{cor}
	Corollary \ref{cor:convexintegralbound}, and hence its Young measure version, can be formulated also for sets $K$ that vary uniformly continuously with respect to the Hausdorff distance $d_H$ as defined in (\ref{eq:kuniformlycont}) below. More precisely, we have the following:
	\begin{cor}\label{cor:varyingk}
		Let $\Omega\subset \R^d$ be open, bounded and let $l=1$ or $l=2$. Let $F\colon \R^k\to \R$ be a non-negative, homogeneous, convex function and let $(K_x)_{x\in\Omega}$ be a family of compact, convex sets in $\R^k$ with $|K_x|_{\infty}\geq \eta>0$ for all $x\in\Omega$ and some fixed $\eta>0$. Further, assume that $(K_x)$ has the property that for all $\eps>0$ there exists $\delta>0$ such that for all $x_0\in \Omega$
		\begin{align}
			 d_H(K_x,K_{x_0}):=\max\left\{\underset{A\in K_x}{\max}\dist(A,K_{x_0}),\underset{B\in K_{x_0}}{\max}\dist(B,K_x) \right\}\leq \eps\text{\ \ if\ \ }|x-x_0|\leq \delta.\label{eq:kuniformlycont}
		\end{align}
		Further, let $(u_j)\in W^{l,\infty}_{\operatorname{loc}}(\Omega,\R^m)$ be a sequence of functions satisfying $\|D^lu_j\|_{L^{\infty}(\Omega)}\leq M$ and
		\begin{align*}
			u_j&\rightarrow u_0\text{ in }L^1_{\operatorname{loc}}(\Omega,\R^m),\\
			(x\mapsto\dist(\mathcal{B}u_j(x),K_x))&\rightarrow 0\text{ in }L^1(\Omega).
		\end{align*}
		Then there exist functions $w_j\in W^{l,\infty}_{}(\Omega,\R^m)$ such that
		\begin{align*}
			|\{u_j\neq w_j \}|&\rightarrow 0,\\
			\|(x\mapsto\dist(\mathcal{B}w_j(x),K_x)\|_{L^{\infty}(\Omega)}&\rightarrow 0,\\
			\underset{j\rightarrow\infty}{\limsup}\int\limits_{\Omega}^{}F(\mathcal{B}w_j)\dx&\leq \underset{j\rightarrow\infty}{\limsup}\int\limits_{\Omega}^{}F(\mathcal{B}u_j)\dx<\infty.
		\end{align*}
	\end{cor}	
	\begin{proof}
		Observe that convexity and compactness of $K_x$ imply
		\begin{align}
			\dist(\mathcal{B}u_j(x),K_{x_0})\leq \dist(\mathcal{B}u_j(x),K_x)+d_H(K_x,K_{x_0})\label{eq:distestimate}
		\end{align}
		for all $x,x_0\in\Omega$. We introduce a segmentation of $\R^d$ by open cubes $(\tilde{Q}_1^n)_{n\in\N}$ of side-length one with corners lying in $\Z^d$. Then the next family of cubes $\left(\tilde{Q}_{\frac{1}{2}}^n\right)$ arises from the former by bisecting the edges. So, one cube $\tilde{Q}^n_1$ contains $2^d$ cubes of the next generation. Continuing leads to a dyadic segmentation into cubes. Note that for all $N\in\N$ we have $\R^d\backslash \left(\underset{n\in\N}{\bigcup}\tilde{Q}_{2^{-N}}^n\right)$ is a null set.\\
		By (\ref{eq:kuniformlycont}) for all $i\in \N$ choose $N_i\in \N$ such that $d_H(K_x,K_y)\leq \frac{1}{i}$ for all $x,y\in \Omega$ with $|x-y|\leq \sqrt{d}2^{-N_i}$. Moreover, for all $n\in \N$ choose $x^n_i$ as the center point of $\tilde{Q}_{2^{-N_i}}^n$. Now change the families of cubes to $\left(Q_{2^{-N_i}}^n\right):=\left(\tilde{Q}_{2^{-N_i}}^n\cap \Omega\right)$ and change the points $x_i^n$ (without renaming) to an arbitrary point in $Q_{2^{-N_i}}^n$ if $x_i^n\notin \Omega$. Thus, for all cubes $Q_{2^{-N_i}}^n$ it holds that $u_j\big|_{Q_{2^{-N_i}}^n}\in W^{l,\infty}\left(Q_{2^{-N_i}}^n,\R^m\right)$ with
		\begin{align*}
			\|D^lu_j\|_{L^{\infty}\left(Q_{2^{-N_i}}^n\right)}\leq M\leq\frac{M}{\eta}|K_{x^n_i}|_{\infty}.
		\end{align*}
		We also have by (\ref{eq:distestimate})
		\begin{align*}
			\int\limits_{Q_{2^{-N_i}}^n}^{}\dist(\mathcal{B}u_j(x),K_{x_i^n})\dx\leq \int\limits_{Q_{2^{-N_i}}^n}^{}\dist(\mathcal{B}u_j(x),K_x)+d_H(K_x,K_{x_i^n})\dx\leq \int\limits_{Q_{2^{-N_i}}^n}^{}\dist(\mathcal{B}u_j(x),K_x)\dx+|\Omega|\frac{1}{i}.
		\end{align*}
		Hence,
		\begin{align*}
			\dist\left(\mathcal{B}u_j(\cdot),(K_{x_i^n})_{|\Omega|\frac{1}{i}}\right)\rightarrow 0\text{ in }L^1\left(Q_{2^{-N_i}}^n\right).
		\end{align*}
		The $L^1\left(Q_{2^{-N_i}}^n\right)$-convergence of $u_j$ to $u_0$ is clear.\\
		Thus, on every $Q_{2^{-N_i}}^n$ we apply Corollary  \ref{cor:convexintegralbound} to obtain a sequence $w_j^{i,n}\in W^{l,\infty}_{}\left(Q_{2^{-N_i}}^n,\R^m\right)$ with
		\begin{align}
			\left|\{u_j\neq w_j^{i,n} \}\cap Q_{2^{-N_i}}^n\right|&\overset{j\rightarrow\infty}{\rightarrow} 0,\nonumber\\
			\left\|\dist\left(\mathcal{B}w_j^{i,n},(K_{x_i^n})_{|\Omega|\frac{1}{i}}\right)\right\|_{L^{\infty}\left(Q_{2^{-N_i}}^n\right)}&\overset{j\rightarrow\infty}{\rightarrow} 0,\nonumber\\
			\underset{j\rightarrow \infty}{\limsup}\int\limits_{Q^n_{2^{-N_i}}}^{}F\left(\mathcal{B}w_j^{i,n} \right)\dx &\leq \underset{j\rightarrow \infty}{\limsup}\int\limits_{Q^n_{2^{-N_i}}}^{}F\left(\mathcal{B}u_j \right)\dx.\label{eq:wrongintegralbound}
		\end{align}
		Moreover, for neighboring cubes the so constructed sequences agree near the common boundary and equal $u_0$ there. Hence, the ensemble $w_j^{i,n}$ defines for fixed $i,j$ a function $w_j^i\in W^{l,\infty}(\Omega,\R^m)$.\\
		As $\Omega$ is bounded we only have finitely many $Q_{2^{-N_i}}^n$ for each fixed $i$. So, for all $i\in \N$ there exists $j_i\in\N$ such that $|\{u_j\neq w_j^i \}|\leq \frac{1}{i}$ and $\left\|\dist\left(\mathcal{B}w_j^{i},(K_{x_i^n})_{|\Omega|\frac{1}{i}}\right)\right\|_{L^{\infty}\left(Q_{2^{-N_i}}^n\right)}\leq \frac{1}{i}$ for every $n$ and for all $j\geq j_i$. The estimate in (\ref{eq:wrongintegralbound}) does not lead to an estimate for integral bounds on $\Omega$, as the limit superior on the right-hand side may not commute with the arising sum over $n$. But luckily in the proof of Corollary \ref{cor:convexintegralbound} below we will establish the inequality (\ref{eq:integralboundfinal}) which compares the integrals for every $j$ and not just in the limit superior. The latter estimate shows that we can in fact choose $j_i$ such that also
		\begin{align*}
			\int\limits_{Q^n_{2^{-N_i}}}^{}F\left(\mathcal{B}w_{j}^{i,n} \right)\dx\leq e^{\frac{1}{i}}\int\limits_{Q^n_{2^{-N_i}}}^{}F\left(\mathcal{B}u_j \right)\dx+\frac{1}{i}\left|Q^n_{2^{-N_i}}\right|
		\end{align*}
		holds for all $n$ and for all $j\geq j_i$.\\
		Without loss of generality we may assume that $j_i$ is strictly increasing. For all $i\in \N$ we define $w_j:=w_j^i$ for $j_i\leq j<j_{i+1}$ and $w_j=u_0$ for $j<j_1$. Then
		\begin{align*}
			|\{u_j\neq w_j\}|&\overset{j\rightarrow\infty}{\rightarrow} 0,
		\end{align*}
		and for all $j\in [j_i,j_{i+1})$ and almost every $x\in \Omega$ it holds that $x\in Q_{2^{-N_i}}^n$ for some $n$ and
		\begin{align*}
			\dist(\mathcal{B}w_j(x),K_x)\leq \dist\left(\mathcal{B}w_j(x),(K_{x_i^n})_{|\Omega|\frac{1}{i}}\right)+d_H\left(K_x,(K_{x_i^n})_{|\Omega|\frac{1}{i}}\right)\leq (2+|\Omega|)\frac{1}{i}.
		\end{align*}
		For the integral bound it holds that
		\begin{align*}
			\int\limits_{\Omega}^{}F(\mathcal{B}w_j)\dx\leq e^{\frac{1}{i}}\int\limits_{\Omega}^{}F(\mathcal{B}u_j)\dx+|\Omega|\frac{1}{i}
		\end{align*}
		for all $j\in [j_i,j_{i+1})$. As $i$ grows with $j$, this finishes the proof.
	\end{proof}
	\begin{rem}
		One can generalize the results of this paper to sequences of the form $(f+\mathcal{B}u_j)$, where $(u_j)$ is as before and $f$ is locally Lipschitz on $\Omega$ with $\mathcal{A}f=0$. Here $\mathcal{B}$ has constant rank and $\mathcal{A}$ is a constant rank homogeneous differential operator with $\ker(\mathcal{A})=\operatorname{im}(\mathcal{B})$. In this case the constructed sequence is then also of the form $(f+\mathcal{B}w_j)$. For this one needs to assume a bound on the radius in Lemma \ref{lemm:muellerlemma5} below, which has to be dealt with in the succeeding auxiliary results. Although this requires some effort, no substantial new ideas come in. Note also that in view of Lemma 5 in \cite{R18} the function $f$ can be rewritten as $f=\mathcal{B}\varphi$ if $f$ is a function of high regularity on the whole space $\R^d$ or has zero mean on the torus. So, this boils down to the situation discussed in Theorem \ref{theo:mueller4}.
	\end{rem}
	\section{Auxiliary results and proof of the main theorems}\label{sect:proofsinrd}
	This section is dedicated to the presentation of a detailed proof of our main result.
	\begin{lem}\label{lemm:muellerlemma5}
		Let $u\in W^{l,\infty}(B_r(a),\R^m)$ with $\|D^lu\|_{L^{\infty}(B_r(a))}\leq M|K|_{\infty}$ and
		\begin{align*}
			\theta\geq \frac{1}{|K|_{\infty}}\underset{B_r(a)}{\dashint}\dist(\mathcal{B}u,K)\dx
		\end{align*}
		for some $\theta<90^{-(d+1)}$. Define $\gamma:=\theta^{\frac{1}{1+d}}(1+C_1M)|K|_{\infty}$, where $C_1$ is a fixed constant that only depends on $\mathcal{B},d$.\\
		Then there exists $\tilde{u}\in W^{l,\infty}(B_r(a),\R^m)$ such that $u=\tilde{u}$ on $B_r(a)\backslash B_{\frac{7}{8}r}(a)$ and
		\begin{align*}
			\int\limits_{B_r(a)}^{}\dist(\mathcal{B}\tilde{u},K_{\gamma})\dx\leq \left(1+10\cdot\theta^{\frac{1}{1+d}}\right)\int\limits_{B_r(a)\backslash B_{\frac{r}{2}}(a)}^{}\dist(\mathcal{B}u,K)\dx.
		\end{align*}
	\end{lem}
	Moreover,
	\begin{align*}
		\|D^l\tilde{u}\|_{L^{\infty}(B_r(a))}\leq \left(1+C_1\theta^{\frac{1}{d+1}}\right)|K|_{\infty}M.
	\end{align*}
	\begin{proof}
		We follow the basic strategy of proving Lemma 5 in \cite{M99}. But note that our regularization of the function $u$ will be different, which enables us to go to order $l=2$ under the additional assumption $\|D^lu\|_{L^{\infty}(B_r(a))}\leq M|K|_{\infty}$.\\
		Using the rescaling
		\begin{align*}
			u(\cdot)\rightarrow \frac{1}{|K|_{\infty}}\frac{1}{r^l}u(r(\cdot+a)),\ \tilde{u}(\cdot)\rightarrow \frac{1}{|K|_{\infty}}\frac{1}{r^l}\tilde{u}(r(\cdot+a)),\ K\rightarrow \frac{K}{|K|_{\infty}},\ \gamma\rightarrow \frac{\gamma}{|K|_{\infty}}
		\end{align*}
		we assume that $|K|_{\infty}=1$, $a=0$, and $r=1$. Note that this rescaling argument is the reason why we may only assume a uniform bound on the highest derivative.\\
		In the following we use the notation $B:=B_{1}(0)$ and $B_{t}:=B_{t}(0)$ for all $t>0$. Define the function
		\begin{align*}
			\tilde{u}(x):=\begin{cases}
			\underset{B_{\rho(x)}(x)}{\dashint}u(y)\dy,\,& x\in B_{\frac{7}{8}}\\
			u(x),& x\in B\backslash B_{\frac{7}{8}}
			\end{cases}=\begin{cases}
			\underset{B}{\dashint}u(x+\rho(x)y)\dy,\,& x\in B_{\frac{7}{8}}\\
			u(x),& x\in B\backslash B_{\frac{7}{8}}
			\end{cases},
		\end{align*}
		where $\rho\in C^{\infty}(B,\R) $ is a radially symmetric function such that
		\begin{align*}
			\rho(x)&=0,\,\forall x\in B\backslash B_{\frac{7}{8}},\\
			0< \rho(x)&\leq \varepsilon,\,\forall x\in B_{\frac{7}{8}},\\
			\rho(x)&=\eps,\,\forall x\in B_{\frac{5}{8}},\\
			|D\rho(x)|&\leq 9\eps,\\
			|D^2\rho(x)|&\leq 65\eps,
		\end{align*}
		where $\eps<\frac{1}{90}$ will be chosen later.\\
		Consider for every $y\in B$ the map
		\begin{align*}
			x\mapsto \varphi_y(x):=x+\rho(x)y.
		\end{align*}
		Since
		\begin{align}
			|D\varphi_y|(x)=|\mathds{1}+D\rho(x)\otimes y|\geq |\det(\mathds{1})(1+y^{\operatorname{T}}\cdot D\rho(x))|\geq 1-|y||D\rho(x)|\geq 1-9\eps>\frac{9}{10}\label{eq:varphicalc}
		\end{align}
		for all $x\in \R^d$ and all fixed $y\in B$, the map $\varphi_y$ is a local diffeomorphism. Noticing that $|\varphi_y(x)|\rightarrow \infty$ as $|x|\rightarrow \infty$, we infer by Hadamard's theorem that we even have a global smooth inverse $\varphi_y^{-1}$.\\
		For all fixed $y\in B$ the map $x\mapsto u(\varphi_y(x))$ is differentiable a.e. as $u$ is differentiable almost everywhere. In particular, for fixed $y$ and for a.e. $x\in B$ one calculates
		\begin{align*}
			\partial_{x_i}u(x+\rho(x)y)=(Du)(x+\rho(x)y)\cdot(e_i+y\partial_i\rho(x)), 
		\end{align*}
		with $i=1,...,d$. In the case $l=1$ observe that $u\in W^{1,\infty}(B)\hookrightarrow H^1(B)$, hence so is $x\mapsto u(\varphi_y(x))$. Thus, by the Divergence Theorem and Fubini's theorem we obtain for all $\psi\in C_c^{\infty}(B)$
		\begin{align*}
			-\int\limits_{\R^d}^{}\underset{B_{\rho(x)}(x)}{\dashint}u(y)\dy\, \partial_{x_i}\psi(x)\dx=\int\limits_{\R^d}^{}\underset{B}{\dashint}(Du)(x+\rho(x)y)\cdot(e_i+y\partial_i\rho(x))\dy\,\psi(x)\dx.
		\end{align*}
		In the case $l=2$ we argue similarly but with $x\mapsto u(\varphi_y(x))$ replaced by $x\mapsto (Du)(\varphi_y(x))\cdot(e_i+y\partial_i\rho(x))$. Hence, for fixed $y$ and a.e. $x$ it holds that
		\begin{align*}
			&\partial_{x_j}\big((Du)(x+\rho(x)y)\cdot(e_i+y\partial_i\rho(x))\big)\\
			=&(e_j+y\partial_j\rho(x))^{\operatorname{T}}\cdot (D^2u)(x+\rho(x)y)\cdot (e_i+y\partial_i\rho(x))+(Du)(x+\rho(x)y)\cdot y\partial_{i,j}\rho(x),
		\end{align*}		
		for $i,j=1,...,d$, and calculating the second order weak derivatives yields for all $\psi\in C_c^{\infty}(B)$
		\begin{align*}
			&\int\limits_{\R^d}^{}\underset{B_{\rho(x)}(x)}{\dashint}u(y)\dy\,\partial_{x_i}\partial_{x_j}\psi(x)\dx\\
			=&\int\limits_{\R^d}^{}\underset{B}{\dashint}(e_j+y\partial_j\rho(x))^{\operatorname{T}}\cdot (D^2u)(x+\rho(x)y)\cdot (e_i+y\partial_i\rho(x))+(Du)(x+\rho(x)y)\cdot y\partial_{i,j}\rho(x)\dy\,\psi(x)\dx.
		\end{align*}
		Now, for $x\in B_{\frac{5}{8}}$ we estimate 
		\begin{align}	
			\dist(\mathcal{B}\tilde{u},K)\leq \underset{B_{\eps}(x)}{\dashint}\dist(\mathcal{B}u(y),K)\dy\leq \frac{1}{\eps^{d}}\underset{B}{\dashint}\dist(\mathcal{B}u(y),K)\dy\leq \frac{1}{\eps^{d}}\theta.\label{eq:innerestimate}
		\end{align}
		\textbf{Claim:} It holds that
		\begin{align*}
			\int\limits_{B_{\frac{7}{8}}\backslash B_{\frac{5}{8}}^{}}\dist(\mathcal{B}\tilde{u}(x),K_{\gamma})\dx\leq (1+10\eps)\int\limits_{B_{\frac{7}{8}}\backslash B_{\frac{1}{2}}^{}}\dist(\mathcal{B}u(z),K)\dz.
		\end{align*}
		In order to prove this claim observe that $\rho(x)>0$ for all $x\in B_{\frac{7}{8}}\backslash B_{\frac{5}{8}}$, thus $y\mapsto x+\rho(x)y$ is a diffeomorphism from $B$ to $B_{\rho(x)}(x)$. Now we need to separate the cases $l=1$ and $l=2$ again. For $l=1$ and $x\in B_{\frac{7}{8}}\backslash B_{\frac{5}{8}}$ we estimate for $i=1,...,d$
		\begin{align*}
			\left|\underset{B}{\dashint}(Du)(x+\rho(x)y)\cdot y\partial_i\rho(x)\dy\right|\leq \underset{B}{\dashint}9M\varepsilon|y|\dy\leq 9M\varepsilon.
		\end{align*}
		Choosing $\eps:=\theta^{\frac{1}{1+d}}$ and $C_1:=9\sum\limits_{i=1}^{d}|B^i|$ yields
		\begin{align*}
			\frac{\theta}{\eps^{d}}+\sum\limits_{i=1}^{d}|B^{i}|9M\eps=(1+C_1M)\eps=\gamma.
		\end{align*}
		Since $K$ is convex, the distance function $z\mapsto \dist(z,K)$ is convex. Moreover, $\dist(z+w,K)\leq \dist(z,K)+|w|$ for all $z,w\in \R^m$. Hence, we estimate
		\begin{align*}
			\int\limits_{B_{\frac{7}{8}}\backslash B_{\frac{5}{8}}^{}}\dist(\mathcal{B}\tilde{u}(x),K_{\gamma})\dx\leq & \int\limits_{B_{\frac{7}{8}}\backslash B_{\frac{5}{8}}^{}}\left(\dist\left(\sum\limits_{i=1}^{d}B^{i}\underset{B}{\dashint}(Du)(x+\rho(x)y)\cdot(e_i+y\partial_i\rho(x))\dy,K\right)-\gamma\right)^+\dx\\
			\leq & \int\limits_{B_{\frac{7}{8}}\backslash B_{\frac{5}{8}}^{}}\left(\dist\left(\underset{B}{\dashint}(\mathcal{B}u)(x+\rho(x)y),K\right)+\sum\limits_{i=1}^{d}|B^{i}|9M\eps -\gamma\right)^+\dx\\
			\leq & \int\limits_{B_{\frac{7}{8}}\backslash B_{\frac{5}{8}}^{}}\underset{B}{\dashint}\dist((\mathcal{B}u)(x+\rho(x)y),K)\dy\dx.
		\end{align*}
		Note that the estimate (\ref{eq:varphicalc}) yields
		\begin{align}
			\frac{1}{|D\varphi_y|(x)}\leq \frac{1}{1-9\eps}\leq 1+10\eps\label{eq:varphiestimate}
		\end{align}
		for all $x\in B_{\frac{7}{8}}\backslash B_{\frac{5}{8}}$, as $\eps< \frac{1}{90}$. Furthermore, the definition of $\rho$ implies
		\begin{align}
			\varphi_y\left(B_{\frac{7}{8}}\backslash B_{\frac{5}{8}}\right)\subset B_{\frac{7}{8}}\backslash B_{\frac{1}{2}}.\label{eq:varphiimage}
		\end{align}
		Thus, using Fubini's theorem and the substitution $z=\varphi_y(x)$ for every fixed $y$ gives
		\begin{align*}
			\int\limits_{B_{\frac{7}{8}}\backslash B_{\frac{5}{8}}^{}}\underset{B}{\dashint}\dist((\mathcal{B}u)(x+\rho(x)y),K)\dy\dx\leq (1+10\eps)\int\limits_{B_{\frac{7}{8}}\backslash B_{\frac{1}{2}}}^{}\dist(\mathcal{B}u(z),K)\dz.
		\end{align*}
		This proves the claim for $l=1$.\\
		In the case $l=2$ we estimate for $x\in B_{\frac{7}{8}}\backslash B_{\frac{5}{8}}$ and $i,j=1,...,d$
		\begin{align*}
			\left|\underset{B}{\dashint}y\partial_i\rho(x)\cdot(D^2u)(x+\rho(x)y)\cdot y\partial_j\rho(x)\dy\right|\leq \underset{B}{\dashint}81M\varepsilon^2|y|^2\dy\leq 9M\varepsilon
 		\end{align*}
		and
		\begin{align*}
			\left|\underset{B}{\dashint}e_i\cdot(D^2u)(x+\rho(x)y)\cdot y\partial_j\rho(x)\dy\right|\leq \underset{B}{\dashint}9M\varepsilon|y|\dy\leq 9M\varepsilon,
		\end{align*}
		and similarly for $\underset{B}{\dashint}y\partial_i\rho(x)\cdot(D^2u)(x+\rho(x)y)\cdot e_j\dy$. For the term involving only the gradient of $u$ we need to rewrite the expression in terms of the second order derivatives. This can be done since $Du$ is a locally Lipschitz - and hence an absolutely continuous - function. More precisely, we rewrite and estimate for $x\in B_{\frac{7}{8}}\backslash B_{\frac{5}{8}}$ and $i,j=1,...,d$ as follows
		\begin{align*}
			\left|\underset{B}{\dashint}(Du)(x+\rho(x)y)\cdot y\partial_{i}\partial_j\rho(x)\dy\right|=&\left|\underset{B}{\dashint}\frac{1}{2}((Du)(x+\rho(x)y)-(Du)(x-\rho(x)y))\cdot y\partial_{i}\partial_j\rho(x)\dy\right|\\
			\leq & \frac{1}{2}\underset{B}{\dashint}\int\limits_{-\rho(x)|y|}^{\rho(x)|y|}\left|(D^2u)\left(x+t\frac{y}{|y|}\right)\right|\dd t\cdot |y||\partial_{i}\partial_j\rho(x)|\dy\\
			\leq & \underset{B}{\dashint}M\rho(x)65\eps\dy\\
			\leq & 9M\eps. 
		\end{align*}
		The previous calculation only works for $l=2$. Hence, for generalizing to higher order differential operators one has to find more suitable ways of regularizing $u$ or come up with different strategies of proofs.\\
		Similarly as in the case $l=1$ we choose $\eps:=\theta^{\frac{1}{1+d}}$. Set $C_1:=36\sum\limits_{i,j=1}^{d}|B^{ij}|$ to obtain
		\begin{align*}
			\frac{\theta}{\eps^{d}}+\sum\limits_{i,j=1}^{d}|B^{ij}|4\cdot 9M\eps=(1+C_1M)\eps=\gamma.
		\end{align*}
		Proceeding as in the case of $l=1$ and using the above estimates as well as (\ref{eq:varphiestimate}) and (\ref{eq:varphiimage}) we estimate
		\begin{align*}
			\int\limits_{B_{\frac{7}{8}}\backslash B_{\frac{5}{8}}^{}}\dist(\mathcal{B}\tilde{u}(x),K_{\gamma})\dx \leq & \int\limits_{B_{\frac{7}{8}}\backslash B_{\frac{5}{8}}^{}}\left(\dist\left(\underset{B}{\dashint}(\mathcal{B}u)(x+\rho(x)y),K\right)+\sum\limits_{i,j=1}^{d}|B^{ij}|4\cdot 9M\eps -\gamma\right)^+\dx\\
			\leq & \int\limits_{B_{\frac{7}{8}}\backslash B_{\frac{5}{8}}^{}}\underset{B}{\dashint}\dist((\mathcal{B}u)(x+\rho(x)y),K)\dy\dx\\
			\leq & (1+10\eps)\int\limits_{B_{\frac{7}{8}}\backslash B_{\frac{1}{2}}}^{}\dist(\mathcal{B}u(z),K)\dz.
		\end{align*}
		This concludes the proof of the claim.\\
		\\
		From (\ref{eq:innerestimate}) we obtain
		\begin{align*}
			\dist(\mathcal{B}\tilde{u}(x),K)<\gamma
		\end{align*}
		for all $x\in B_{\frac{5}{8}}$. Therefore,
		\begin{align*}
			\int\limits_{B}^{}\dist(\mathcal{B}\tilde{u},K_{\gamma})(x)\dx&= \int\limits_{B\backslash B_{\frac{7}{8}}}^{}\dist(\mathcal{B}\tilde{u},K_{\gamma})(x)\dx+\int\limits_{B_{\frac{7}{8}}\backslash B_{\frac{5}{8}}}^{}\dist(\mathcal{B}\tilde{u},K_{\gamma})(x)\dx+\int\limits_{B_{\frac{5}{8}}}^{}\dist(\mathcal{B}\tilde{u},K_{\gamma})(x)\dx\\
			&\leq \int\limits_{B\backslash B_{\frac{7}{8}}}^{}\dist(\mathcal{B}u,K)(x)\dx+(1+10\cdot\eps)\int\limits_{B_{\frac{7}{8}}\backslash B_{\frac{1}{2}}}^{}\dist(\mathcal{B}u,K)(x)\dx+0\\
			&\leq \left(1+10\cdot\theta^{\frac{1}{1+d}}\right)\int\limits_{B\backslash B_{\frac{1}{2}}}^{}\dist(\mathcal{B}u,K)(x)\dx.
		\end{align*}
		Moreover, for $l=2$ we estimate
		\begin{align*}
			\left\|D^2\tilde{u}\right\|_{L^{\infty}(B)}=&\max\left\{\left\|D^2\tilde{u}\right\|_{L^{\infty}\left(B\backslash B_{\frac{7}{8}}\right)},\left\|D^2\tilde{u}\right\|_{L^{\infty}\left(B_{\frac{7}{8}}\right)} \right\}\\
			\leq &\max\left\{\left\|D^2u\right\|_{L^{\infty}\left(B\backslash B_{\frac{7}{8}}\right)},\left\| \underset{B_{\rho(x)}(x)}{\dashint}\left|D^2u\right|(y)\dy\right\|_{L^{\infty}\left(B_{\frac{7}{8}}\right)}+\sum\limits_{i,j=1}^{d}|B^{ij}|4\cdot 9M\eps \right\}\\
			\leq & \left(1+C_1\eps\right)M.
		\end{align*}
		The corresponding estimate for $l=1$ follows similarly.
	\end{proof}
	The following result is the analogue of Lemma 6 combined with the remark thereafter in \cite{M99}.
	\begin{lem}\label{lemm:muellerlemma6}
		There exist positive constants $\alpha(d)<1$, $C_2(d)<\frac{1}{90}$ with the following property. Let $u\in W^{l,\infty}_{\operatorname{loc}}(\R^d,\R^m)$ be such that $\|D^lu\|_{L^{\infty}(\R^d,\R^m)}\leq M|K|_{\infty}$, and let $\gamma\in (0,C_2(d)(1+C_1M)|K|_{\infty})$ as well as
		\begin{align*}
			\lambda:=\frac{1}{|K|_{\infty}}\int\limits_{\R^d}^{}\dist(\mathcal{B}u,K)\dx.
		\end{align*}
		Then there exists a function $\tilde{u}\in W^{l,\infty}_{\operatorname{loc}}(\R^d,\R^m)$ with
		\begin{align*}
			\frac{1}{|K|_{\infty}}\int\limits_{\R^d}^{}\dist(\mathcal{B}\tilde{u},K_{\gamma})\dx\leq \alpha(d)\lambda,
		\end{align*}
		and such that
		\begin{align*}
			|\{u\neq \tilde{u}\}|\leq 2^d\lambda\left(\frac{(1+C_1M)|K|_{\infty}}{\gamma}\right)^{d+1}
		\end{align*}
		and
		\begin{align*}
			\|D^l\tilde{u}\|_{L^{\infty}(\R^d)}\leq |K|_{\infty}M+\gamma.
		\end{align*}
		Moreover, if $\mathcal{B}u\in K$ on $\R^d\backslash V$, then
		\begin{align*}
			\{u\neq \tilde{u} \}\subset V_{\rho}=\{x\,:\,\dist(x,V)\leq \rho \}
		\end{align*}
		with
		\begin{align*}
			\rho=C_3((1+C_1M)|K|_{\infty})^{\left(\frac{1}{d}+1\right)}\frac{\lambda^{\frac{1}{d}}}{\gamma^{\left(\frac{1}{d}+1\right)}},
		\end{align*}
		where $C_3=C_3(d)$ is some constant depending only on the dimension.
	\end{lem}
	\begin{proof}
		The main part of the proof is essentially identical to the proof of Lemma 6 in \cite{M99} except for the obvious changes in the numerical values of some constants involved. The reason for this is that the structure of the differential operator under consideration plays no role.\\
		The only thing left to prove is our statement about the norm of the highest order derivative of $\tilde{u}$. For that observe that $\tilde{u}=u$ on $\R^d\backslash A$, where $A$ is a union of disjoint balls as in Lemma 6 in \cite{M99} on which we applied Lemma \ref{lemm:muellerlemma5}. In particular,		
		\begin{align*}
			\|D^l\tilde{u}\|_{L^{\infty}(A)}\leq \left(1+C_1\theta^{\frac{1}{d+1}}\right)M,
		\end{align*}
		where $\theta:=\left(\frac{\gamma}{1+C_1M} \right)^{d+1}$. Thus, we estimate
		\begin{align*}
			\|D^l\tilde{u}\|_{L^{\infty}(\R^d)}&\leq \max\{ \|D^lu\|_{L^{\infty}(\R^d\backslash A)},\|D^l\tilde{u}\|_{L^{\infty}(A)}\}\leq \max\left\{M,\left(1+C_1\theta^{\frac{1}{d+1}}\right)M\right\}\\
			&=M+\frac{C_1M\gamma}{1+C_1M}\leq M+\gamma. 
		\end{align*}
	\end{proof}
	We now prove the analogue of Theorem 7 and Corollary 8 in \cite{M99}.
	\begin{prop}\label{prop:muellerprop7undcor8}
		Let $u\in W^{l,\infty}_{\operatorname{loc}}(\R^d,\R^m)$ be such that $\|D^lu\|_{L^{\infty}(\R^d,\R^m)}\leq M|K|_{\infty}$ and suppose $\gamma \in (0,C_2(d)(1+C_1M)|K|_{\infty})$ as well as
		\begin{align*}
			\lambda&:=\frac{1}{|K|_{\infty}}\int\limits_{\R^d}^{}\dist(\mathcal{B}u,K)\dx<\infty,
		\end{align*}
		where $C_1,C_2$ are the constants from Lemma \ref{lemm:muellerlemma5} and Lemma \ref{lemm:muellerlemma6}.\\
		Then there exists $g\in W^{l,\infty}(\R^d,\R^m)$ such that
		\begin{align*}
			\mathcal{B}g&\in K_{\gamma}\text{ a.e. on }\R^d,\\
			\|D^lg\|_{L^{\infty}(\R^d)}&\leq M|K|_{\infty}+\gamma,
		\end{align*}
		and
		\begin{align}
			|\{u\neq g \}|\leq C_4(M,d)\lambda\left(\frac{|K|_{\infty}}{\gamma}\right)^{d+1}\label{eq:neqestimate}
		\end{align}
		for some constant $C_4=C_4(M,d)$ depending on $M$ and the dimension $d$.\\
		Moreover, if $\mathcal{B}u\in K$ on $\R^d\backslash V$, then
		\begin{align*}
			\{u\neq g \}\subset V_{\rho},
		\end{align*}
		where $\rho= C_5(M,d)|K|_{\infty}^{\left(\frac{1}{d}+1\right)}\frac{\lambda^{\frac{1}{d}}}{\gamma^{\left(\frac{1}{d}+1\right)}}$ for some constant $C_5=C_5(M,d)$. In fact, in the latter case it holds that $\lambda=\frac{1}{|K|_{\infty}}\int\limits_{V}^{}\dist(\mathcal{B}u,K)\dx$, and it would have been enough to suppose $\|D^lu\|_{L^{\infty}(V_{\rho})}\leq M|K|_{\infty}$ as the values of $u$ outside $V_{\rho}$ do not enter the construction at any place.
	\end{prop}
	\begin{proof}
		The strategy of the proof of Theorem 7 and Corollary 8 in \cite{M99} is adapted here. In fact, the main issue is that one needs to keep track of the norm of $D^lu$.\\
		Again by rescaling assume that $|K|_{\infty}=1$. Define
		\begin{align*}
			K_0&=K,\\
			M_i&=|K_i|_{\infty},\\
			\gamma_i&=\delta\alpha^{\frac{i}{2(d+1)}}M_i,\\
			K_{i+1}&=(K_i)_{\gamma_i},
		\end{align*}
		where $\alpha=\alpha(d)$ is the constant from Lemma \ref{lemm:muellerlemma6} and $\delta>0$ will be chosen later. Observe that
		\begin{align*}
			M_0&=1,\\
			\log\frac{M_{i+1}}{M_i}&=\log\frac{M_i+\gamma_i}{M_i}\leq \delta \alpha^{\frac{i}{2(d+1)}}.
		\end{align*}
		Thus,
		\begin{align*}
			1\leq M_i&\leq e^{\delta\sum\limits_{i=0}^{\infty}\alpha^{\frac{i}{2(d+1)}}}=:\bar{M},\\
			\sum\limits_{i=0}^{\infty}\gamma_i&\leq \delta\sum\limits_{i=0}^{\infty}\alpha^{\frac{i}{2(d+1)}}\bar{M}=:\bar{\gamma}.
		\end{align*}
		Now by successively applying Lemma \ref{lemm:muellerlemma6} we obtain a sequence $(u_i)$ with $u_0=u$. Define
		\begin{align*}
			\lambda_i&:=\frac{1}{|M_i|}\int\limits_{\R^d}^{}\dist(\mathcal{B}u_i,K_i)\dy,\\
			\mu_i&:=|\{u_{i+1}\neq u_i \}|.
		\end{align*}
		Lemma \ref{lemm:muellerlemma6} yields
		\begin{align*}
			\lambda_{i+1}&\leq \alpha\lambda_i,\\
			\|D^lu_{i+1}\|_{L^{\infty}(\R^d)}&\leq M+\sum\limits_{j=0}^{i}\gamma_j\leq M+\bar{\gamma},\\
			\mu_i&\leq 2^d\lambda_i\left(\frac{(1+C_1(M+\bar{\gamma}))\bar{M}}{\gamma_i}\right)^{d+1}.
		\end{align*}
		Hence,
		\begin{align*}
			\lambda_i&\leq \alpha^i\lambda,\\
			\mu_i&\leq 2^d\big((1+C_1(M+\bar{\gamma}))\bar{M}\big)^{d+1}\delta^{-(d+1)}\alpha^{\frac{i}{2}}\lambda.
		\end{align*}
		Since $\sum\limits_{i=0}^{\infty}\mu_i<\infty$, one deduces analogously to the proof of Theorem 7 in \cite{M99} that
		\begin{align}
			u_i\rightarrow g&\text{ in }W^{l,1}_{\operatorname{loc}}(\R^d,\R^m),\label{eq:wloneconvergence}\\
			\mathcal{B}g&\in K_{\bar{\gamma}}\text{ a.e. in }\R^d.\nonumber
		\end{align}
		Moreover, as $\|D^lu_i\|_{L^{\infty}(\R^d)}\leq M+\bar{\gamma}$, there is a subsequence which we still denote $(u_i)$ and $h\in L^{\infty}(\R^d)$ such that
		\begin{align*}
			D^lu_i\overset{*}{\rightharpoonup} h\text{ in }L^{\infty}(\R^d).
		\end{align*}
		Thus, $\|h\|_{L^{\infty}(\R^d)}\leq M+\bar{\gamma}$. Using (\ref{eq:wloneconvergence}) and testing against compactly supported smooth functions yields $h=D^lg$ a.e., and hence
		\begin{align*}
			\|D^lg\|_{L^{\infty}(\R^d)}\leq M+\bar{\gamma}.
		\end{align*}
		We also obtain
		\begin{align*}
			|\{u\neq g\}|\leq \sum\limits_{i=0}^{\infty}\mu_i\leq 2^d\big((1+C_1(M+\bar{\gamma}))\bar{M}\big)^{d+1}\delta^{-(d+1)}\lambda\sum\limits_{i=0}^{\infty}\alpha^{\frac{i}{2}}=:\bar{c}(1+C_1(M+\bar{\gamma}))^{d+1}\bar{M}^{d+1}\delta^{-(d+1)}\lambda.
		\end{align*}
		Similarly as in \cite{M99} we choose $\delta$ such that
		\begin{align}
			\gamma=\bar{\gamma}=\delta\sum\limits_{i=0}^{\infty}\alpha^{\frac{i}{2(d+1)}}e^{\delta\sum\limits_{i=0}^{\infty}\alpha^{\frac{i}{2(d+1)}}}=:\delta\bar{\alpha}e^{\delta\bar{\alpha}}.\label{eq:gammadelta}
		\end{align}
		As $\gamma<C_2(1+C_1M)$, we obtain that $\delta\leq\bar{\alpha}^{-1}C_2(1+C_1M)$, and hence
		\begin{align*}
			\delta\geq \gamma \bar{\alpha}^{-1}e^{-C_2(1+C_1M)}.
		\end{align*}
		Choosing
		\begin{align*}
			C_4(M,d):=\bar{c}\bar{\alpha}^{d+1}\big(1+C_1(M+C_2(1+C_1M))\big)^{d+1}e^{2(d+1)C_2(1+C_1M)}
		\end{align*}
		yields the estimate (\ref{eq:neqestimate}).\\
		Now assume that $\mathcal{B}u\in K$ on $\R^d\backslash V$. Let
		\begin{align*}
			V_i&:=V\cup \{u_i\neq u \},\\
			\rho_i&:=C_3\left(1+C_1\left(M+\sum\limits_{j=0}^{i}\gamma_j\right)\bar{M}\right)^{\left(\frac{1}{d}+1\right)}\frac{\lambda_i^{\frac{1}{d}}}{\gamma_i^{\left(\frac{1}{d}+1\right)}}.
		\end{align*}
		Hence, $\mathcal{B}u_i\in K$ in $\R^d\backslash V_i$. So, Lemma \ref{lemm:muellerlemma6} yields
		\begin{align*}
			V_{i+1}\subset V\cup \{u_i\neq u \}\cup \{u_{i+1}\neq u_i \}\subset (V_i)_{\rho_i}.
		\end{align*}
		The inequalities $\lambda_i\leq \alpha^i\lambda$, $1\leq M_i$ and the definition of $\gamma_i$ yield
		\begin{align*}
			\rho:=\sum\limits_{i=0}^{\infty}\rho_i&\leq (1+C_1(M+\gamma))^{\left(\frac{1}{d}+1 \right)}\frac{\lambda^{\frac{1}{d}}}{\delta^{\left(\frac{1}{d}+1 \right)}}\bar{M}^{\left(\frac{1}{d}+1 \right)}C_3\sum\limits_{i=0}^{\infty}\alpha^{\frac{i}{2d}}\\
			&\leq \frac{\lambda^{\frac{1}{d}}}{\gamma^{\left(\frac{1}{d}+1\right)}}(1+C_1(M+C_2(1+C_1M)))^{\left(\frac{1}{d}+1\right)}e^{2\left(\frac{1}{d}+1 \right)C_2(1+C_1M)}\bar{\alpha}^{\left(\frac{1}{d}+1 \right)}C_3\sum\limits_{i=0}^{\infty}\alpha^{\frac{i}{2d}}\\
			&=:C_5(M,d)\frac{\lambda^{\frac{1}{d}}}{\gamma^{\left(\frac{1}{d}+1\right)}},
		\end{align*}
		where we used (\ref{eq:gammadelta}).
	\end{proof}
	\begin{proof}[Proof of Theorem \ref{theo:mueller4}]
		Due to rescaling we may assume $|K|_{\infty}=1$.\\
		Let $U\subset\subset \Omega$ be open. Since $\|D^lu_j\|_{L^{\infty}(\Omega)}\leq M$ we can extract a subsequence such that $D^lu_{j}\overset{*}{\rightharpoonup} h$ in $L^{\infty}(\Omega)$ for some $h\in L^{\infty}(\Omega)$. As $u_j\rightarrow u_0$ in $L^1(U)$ testing against compactly supported smooth functions yields that $h=D^lu_0$ on $U$. Uniqueness of the limit yields that $D^lu_{j}\overset{*}{\rightharpoonup} D^lu_0$ in $L^{\infty}(\Omega)$. In particular, also $\mathcal{B}u_{j}\rightharpoonup \mathcal{B}u_0$ in $L^{1}(\Omega)$. Similarly as in \cite{M99} using Mazur's and Fatou's lemmas as well as $\dist(\mathcal{B}u_j,K)\rightarrow 0$ in $L^1(\Omega)$ implies that $\mathcal{B}u_0\in K$ a.e. in $U$, hence a.e. in $\Omega$ as $U$ was arbitrary.\\
		Now let $V\subset \subset U\subset\subset \Omega$ and $\varphi\in C_c^{\infty}(V)$ with $0\leq \varphi\leq 1$. Define
		\begin{align*}
			w_j:=\varphi u_j+(1-\varphi)u_0.
		\end{align*}
		Thus,
		\begin{align*}
			Dw_j=\varphi Du_j+(1-\varphi)Du_0+(u_j-u_0)\otimes D\varphi.
		\end{align*}
		\textbf{Claim:} It holds that
		\begin{align*}
			\lambda_j:=\int\limits_{\Omega}^{}\dist(\mathcal{B}w_j,K)\dx\rightarrow 0.
		\end{align*}
		To prove this claim we have to consider the cases $l=1$ and $l=2$ separately:\\
		The case $l=1$ is easier. Observe that $\mathcal{B}w_j\in K$ on $\Omega\backslash V$ and
		\begin{align*}
			\lambda_j=\int\limits_{\Omega}^{}\dist(\mathcal{B}w_j,K)\dx \leq \int\limits_{V}^{}\dist(\mathcal{B}u_j,K)\dx +\sum\limits_{i=1}^{d}|B^i|\int\limits_{V}^{}|u_j-u_0||D\varphi|\dx\rightarrow 0
		\end{align*}
		by the assumptions of the theorem and the properties of the distance function $\dist(\cdot,K)$.\\
		For the case $l=2$ first note that
		\begin{align*}
			D^2w_j=\varphi D^2u_j+(1-\varphi)D^2u_0+R_j,
		\end{align*}
		where
		\begin{align*}
			|R_j|\leq 2|D(u_j-u_0)||D\varphi|+|u_j-u_0||D^2\varphi|.
		\end{align*}
		By the Gagliardo-Nirenberg interpolation inequality we obtain
		\begin{align}
			\|D(u_j-u_0)\|_{L^{\infty}(U)}&\leq c(U)\|D^2(u_j-u_0)\|_{L^{\infty}(U)}^{\frac{d+1}{d+2}}\|u_j-u_0\|_{L^1(U)}^{\frac{1}{d+2}}+C(U)\|u_j-u_0\|_{L^1(U)}\nonumber\\
			&\leq c(U)(2M)^{\frac{d+1}{d+2}}\|u_j-u_0\|_{L^1(U)}^{\frac{1}{d+2}}+C(U)\|u_j-u_0\|_{L^1(U)}\rightarrow 0.\label{eq:gagliardonirenberg}
		\end{align}
		So, in particular, $D(u_j-u_0)\rightarrow 0$ in $L^1(U)$. We have $\mathcal{B}w_j\in K$ on $\Omega\backslash V$, thus
		\begin{align*}
			\lambda_j=\int\limits_{\Omega}^{}\dist(\mathcal{B}w_j,K)\dx \leq \int\limits_{V}^{}\dist(\mathcal{B}u_j,K)\dx +\sum\limits_{i,k=1}^{d}|B^{ik}|\int\limits_{V}^{}2|D(u_j-u_0)||D\varphi|+|u_j-u_0||D^2\varphi|\dx\rightarrow 0.
		\end{align*}
		This proves the claim.\\
		Since $u_j-u_0\rightarrow 0$ in $L^1(U)$, there exists some $M_1=M_1(U)$ such that $\|u_j-u_0\|_{L^1(U)}\leq M_1$. Moreover, by (\ref{eq:gagliardonirenberg}) there exists a constant $M_2=M_2(U)$ such that $\|D(u_j-u_0)\|_{L^{\infty}(U)}\leq M_2$. The Poincar\'e-Wirtinger inequality yields
		\begin{align*}
			\|u_j-u_0\|_{L^{\infty}(U)}\leq \left\|(u_j-u_0)-\underset{U}{\dashint}(u_j-u_0)\dy\right\|_{L^{\infty}(U)}+\underset{U}{\dashint}|u_j-u_0|\dy\leq C_P(U)M_2+\frac{M_1}{|U|}.
		\end{align*}
		Hence,
		\begin{align*}
			\|D^lw_j\|_{L^{\infty}(U)}\leq 
			M_3(U,\varphi).
		\end{align*}
		Let $\delta>0$. The previous discussion together with Proposition \ref{prop:muellerprop7undcor8} implies that there exists $j_0=j_0(U,V,\varphi,\delta)$ such that for all $j\geq j_0$ there exists $g_j\in W^{l,\infty}(\Omega,\R^m)$ with
		\begin{align*}
			\{w_j\neq g_j \}&\subset U,\\
			|\{w_j\neq g_j \}|&<\delta,\\
			\mathcal{B}g_j&\in K_{\delta}\text{ a.e. in }U.
		\end{align*}
		Thus,
		\begin{align*}
			g_j&=u_0\text{ in }\Omega\backslash U,\\
			|\{g_j\neq u_j \}\cap U|&<\delta + |\{\varphi\neq 1 \}\cap V|+|U\backslash V|,\\
			\dist(\mathcal{B}g_j,K)&\leq \delta.
		\end{align*}
		One now finishes the proof exactly as in step 3 in the proof of Theorem 4 in \cite{M99}. 
	\end{proof}
	\begin{proof}[Proof of Corollary \ref{cor:convexintegralbound}]
		We basically need to repeat the steps of the previous proofs of the auxiliary results extending them by appropriate estimates of the convex integral bound corresponding to $F$.\vspace{0.2cm}\\
		\textit{Lemma \ref{lemm:muellerlemma5}}: We prove that additionally to the result of Lemma \ref{lemm:muellerlemma5} we obtain
		\begin{align*}
			\int\limits_{B_r(a)}^{}F(\mathcal{B}\tilde{u})\dx\leq \left(1+10\cdot\theta^{\frac{1}{1+d}}\right)\int\limits_{B_r(a)}^{}F(\mathcal{B}u)\dx+L_M^F|B_r(a)||K|_{\infty}^{h-1}\gamma,
		\end{align*}
		where $h$ is the degree of homogeneity of $F$ and $L_M^F$ is specified below.\\
		For that we use the above rescaling, so we may assume $|K|_{\infty}=1$. Recall from (\ref{eq:varphiestimate}) that for $y\in B$ we have
		\begin{align*}
			\frac{1}{|D\varphi_y|(x)}\leq 1+10\varepsilon
		\end{align*}
		for all $x\in B_{\frac{7}{8}}\backslash B_{\frac{5}{8}}$, which also holds for all $x\in B_{\frac{5}{8}}$. Note that
		\begin{align*}
			\varphi_y\left(B_{\frac{7}{8}} \right)\subset B_{\frac{7}{8}}
		\end{align*}
		for $y\in B$. Hence, using Jensen's inequality, Fubini's theorem, and the substitution $z=\varphi_y(x)=x+\rho(x)y$ for fixed $y$ we obtain
		\begin{align*}
			\int\limits_{B_{\frac{7}{8}}}^{}F\left(\underset{B}{\dashint}(\mathcal{B}u)(x+\rho(x)y)\dy \right)\dx \leq \int\limits_{B_{\frac{7}{8}}}^{}\underset{B}{\dashint}F(\mathcal{B}u)(x+\rho(x)y)\dy \dx \leq (1+10\varepsilon)\int\limits_{B_{\frac{7}{8}}}^{}F(\mathcal{B}u)(z)\dz 
		\end{align*}
		Since $F$ is a convex function, there exists an optimal Lipschitz constant $L^F_M$ for $F$ on $B_{C_1(1+C_1)M}(0)$. The latter ball contains $\mathcal{B}\tilde{u}(x)$ for a.e. $x\in B$ by the bound on $\|D\tilde{u}\|_{L^{\infty}}$ from Lemma \ref{lemm:muellerlemma5} and the definition of $C_1$. Thus,
		\begin{align*}
			\int\limits_{B}^{}F(\mathcal{B}\tilde{u})\dx&=\int\limits_{B\backslash B_{\frac{7}{8}}}^{}F(\mathcal{B}u)\dx+\int\limits_{B_{\frac{7}{8}}\backslash B_{\frac{5}{8}}}^{}F\left(\mathcal{B}\underset{B_{\rho(x)}(x)}{\dashint}u(y)\dy \right)\dx+\int\limits_{B_{\frac{5}{8}}}^{}F\left(\underset{B_{\varepsilon}(x)}{\dashint}\mathcal{B}u(y)\dy \right)\dx\\
			&\leq \int\limits_{B\backslash B_{\frac{7}{8}}}^{}F(\mathcal{B}u)\dx+\int\limits_{B_{\frac{7}{8}}}^{}F\left(\underset{B_{\rho(x)}(x)}{\dashint}\mathcal{B}u(y)\dy \right)\dx+\int\limits_{B_{\frac{7}{8}}\backslash B_{\frac{5}{8}}}^{}L^F_MC_1M\varepsilon\dx\\
			&\leq (1+10\varepsilon)\int\limits_{B}^{}F(\mathcal{B}u)(x)\dx+L_M^F|B|\gamma.
		\end{align*}
		\textit{Lemma \ref{lemm:muellerlemma6}}: We now prove the following claim: If $\mathcal{B}u\in K$ on $\R^d\backslash V$ then
		\begin{align*}
			\int\limits_{V_{\rho}}^{}F(\mathcal{B}\tilde{u})\dx\leq \left(1+10\gamma 	\right)\int\limits_{V_{\rho}}^{}F(\mathcal{B}u)\dx+L_M^F|K|_{\infty}^{h-1}\gamma2^d\lambda\left(\frac{(1+C_1M)|K|_{\infty}}{\gamma}\right)^{d+1}.
		\end{align*}
		For that we rescale and consider the set $A\subset V_{\rho}$ from Lemma \ref{lemm:muellerlemma6} again. We estimate using the previous paragraph
		\begin{align*}
			\int\limits_{V_{\rho}}^{}F(\mathcal{B}\tilde{u})\dx&=\int\limits_{V_{\rho}\backslash A}^{}F(\mathcal{B}u)\dx+\int\limits_{A}^{}F(\mathcal{B}\tilde{u})\dx\\
			&\leq \left(1+10\cdot\theta^{\frac{1}{d+1}}\right)\int\limits_{V_{\rho}\backslash A}^{}F(\mathcal{B}u)\dx+\left(1+10\cdot\theta^{\frac{1}{d+1}}\right)\int\limits_{A}^{}F(\mathcal{B}u)\dx+L_M^F|A|\gamma\\
			&\leq (1+10\gamma)\int\limits_{V_{\rho}}^{}F(\mathcal{B}u)\dx+L_M^F\gamma2^d\lambda\left(\frac{(1+C_1M)|K|_{\infty}}{\gamma}\right)^{d+1}.
		\end{align*}
		\textit{Proposition \ref{prop:muellerprop7undcor8}}: It holds that if $\mathcal{B}u\in K$ on $\R^d\backslash V$ then
		\begin{align*}
			\int\limits_{V_{\rho}}^{}F(\mathcal{B}g)\dx&\leq 	e^{10\frac{\gamma}{|K|_{\infty}}}\int\limits_{V_{\rho}}^{}F(\mathcal{B}u)\dx+\lambda e^{10\frac{\gamma}{|K|_{\infty}}}\gamma|K|_{\infty}^{h-1} L_{e^{C_2(1+C_1M)}}^FC_4(M,d)e^{(h-1)C_2(1+C_1M)} \left(\frac{|K|_{\infty}}{\gamma} \right)^{d+1}.
		\end{align*}
		For showing that let us consider the inductive procedure from the proof of Proposition \ref{prop:muellerprop7undcor8} again. Therein we estimate using the preceding result
		\begin{align*}
			&\int\limits_{V_{\rho}}^{}F(\mathcal{B}u_{i+1})\dx\\
			\leq &\int\limits_{V_{\sum\limits_{j=0}^{i}\rho_j}}^{}F(\mathcal{B}u_{i+1})\dx+\int\limits_{V_{\rho}\backslash V_{\sum\limits_{j=0}^{i}\rho_j}}^{}F(\mathcal{B}u)\dx\\
			\leq 	&(1+10\gamma_i)\int\limits_{V_{\sum\limits_{j=0}^{i}\rho_j}}^{}F(\mathcal{B}u_i)\dx+L_{M_i}^F\gamma_iM_i^{h-1}2^d\lambda_i\left(\frac{(1+C_1(M+\bar{\gamma}))\bar{M}}{\gamma_i}\right)^{d+1}+(1+10\gamma_i)\int\limits_{V_{\rho}\backslash V_{\sum\limits_{j=0}^{i}\rho_j}}^{}F(\mathcal{B}u)\dx\\
			\leq &\prod\limits_{j=0}^{i}(1+10\gamma_j)\int\limits_{V_{\rho}}^{}F(\mathcal{B}u)\dx+\sum\limits_{j=0}^{i}L_{M_j}^F\gamma_jM_j^{h-1}2^d\lambda\big((1+C_1(M+\bar{\gamma}))\bar{M}\big)^{d+1}\delta^{-(d+1)}\alpha^{\frac{j}{2}}\left(\prod\limits_{k=j+1}^{i}(1+10\gamma_{k}) \right)\\
			\leq 	&e^{10\sum\limits_{j=0}^{\infty}\gamma_j}\int\limits_{V_{\rho}}^{}F(\mathcal{B}u)\dx+e^{10\sum\limits_{j=0}^{\infty}\gamma_j}\bar{\gamma}L_{\bar{M}}^F\bar{M}^{h-1}2^d\lambda\big((1+C_1(M+\bar{\gamma}))\bar{M}\big)^{d+1}\delta^{-(d+1)}\sum\limits_{j=0}^{\infty}\alpha^{\frac{j}{2}}\\
			\leq &e^{10\gamma}\int\limits_{V_{\rho}}^{}F(\mathcal{B}u)\dx+e^{10\gamma}\gamma L_{e^{C_2(1+C_1M)}}^FC_4(M,d)e^{(h-1)C_2(1+C_1M)}\lambda \gamma^{-(d+1)},
		\end{align*}
		where we used that $\bar{M}=e^{\delta\bar{\alpha}}\leq e^{C_2(1+C_1M)}$.\\
		Note that we showed in the proof of Proposition \ref{prop:muellerprop7undcor8} that $u_i\rightarrow g$ in $W^{l,1}_{\operatorname{loc}}(\R^d,\R^m)$ which implies $\mathcal{B}u_i\rightarrow \mathcal{B}g$ in $L^1(V_{\rho})$. Therefore
		\begin{align*}
			\int\limits_{V_{\rho}}^{}F(\mathcal{B}g)\dx&\leq \lim\limits_{i\rightarrow \infty}\int\limits_{V_{\rho}}^{}L_{\bar{M}}^F|\mathcal{B}u_i-\mathcal{B}g|\dx+\underset{i\rightarrow\infty}{\liminf}\int\limits_{V_{\rho}}^{}F(\mathcal{B}u_i)\dx\\
			&\leq e^{10\gamma}\int\limits_{V_{\rho}}^{}F(\mathcal{B}u)\dx+e^{10\gamma}\gamma L_{e^{C_2(1+C_1M)}}^FC_4(M,d)e^{(h-1)C_2(1+C_1M)}\lambda \gamma^{-(d+1)}.
		\end{align*}
		\textit{Theorem \ref{theo:mueller4}}: We finish the argument by proving
		\begin{align*}
			\underset{j\rightarrow \infty}{\limsup}\int\limits_{\Omega}^{}F(\mathcal{B}g_j)\dx\leq \underset{j\rightarrow\infty}{\limsup}\int\limits_{\Omega}^{}F(\mathcal{B}u_j)\dx.
		\end{align*}
		In the following we use the notation and the rescaling of the proof of Theorem \ref{theo:mueller4}. Let $U,V$ be open with $V\subset\subset U\subset \subset \Omega$, while noting that $\Omega$ is bounded by the assumptions of Corollary \ref{cor:convexintegralbound}. Further let $\varphi\in C_c^{\infty}(V)$ be such that $0\leq \varphi\leq 1$. We already proved that $\mathcal{B}u_j\overset{*}{\rightharpoonup}\mathcal{B}u_0$ in $L^{\infty}(\Omega)$. For $w_j:=\varphi u_j+(1-\varphi)u_0$ it holds that
		\begin{align*}
				\int\limits_{\Omega}^{}F(\mathcal{B}w_j)\dx=&\int\limits_{\Omega\backslash V}^{}F(\mathcal{B}u_0)\dx+\int\limits_{V}^{}F(\varphi\mathcal{B}u_j+(1-\varphi)\mathcal{B}u_0)\dx\\
				&+\int\limits_{V}^{}F(\varphi\mathcal{B}u_j+(1-\varphi)\mathcal{B}u_0+\tilde{R}_j)-F(\varphi\mathcal{B}u_j+(1-\varphi)\mathcal{B}u_0)\dx,
		\end{align*}
		where $\tilde{R}_j:=\mathcal{B}w_j-(\varphi\mathcal{B}u_j+(1-\varphi)\mathcal{B}u_0)$ satisfies the estimate
		\begin{align*}
			\int\limits_{U}^{}|\tilde{R}_j|\dx\leq C_{\mathcal{B}}\int\limits_{U}^{}|R_j|\dx\overset{j\rightarrow\infty}{\rightarrow}0 
		\end{align*}
		with $C_{\mathcal{B}}:=\sum\limits_{|\alpha|=l}^{}|B^{\alpha}|$ and $R_j$ as in the proof of Theorem \ref{theo:mueller4}. Hence,
		\begin{align*}
			r_j:=\int\limits_{V}^{}F(\varphi\mathcal{B}u_j+(1-\varphi)\mathcal{B}u_0+\tilde{R}_j)-F(\varphi\mathcal{B}u_j+(1-\varphi)\mathcal{B}u_0)\dx
		\end{align*}
		converges to zero for $j\rightarrow \infty$ since the functions inside $F$ are uniformly bounded and $F$ is locally Lipschitz. Thus,
		\begin{align*}
			\int\limits_{\Omega}^{}F(\mathcal{B}w_j)\dx&\leq \int\limits_{\{\varphi=1\}\cap V}^{}F(\mathcal{B}u_j)\dx+\int\limits_{\Omega\backslash V}^{}F(\mathcal{B}u_0)\dx+|\{\varphi\neq 1\}\cap V|\cdot L_M^FC_{\mathcal{B}}M+r_j\\
			&\leq \int\limits_{\Omega}^{}F(\mathcal{B}u_j)\dx+L_M^FC_{\mathcal{B}}M(|\Omega\backslash V|+|\{\varphi\neq 1\}\cap V|)+r_j.
		\end{align*}
		Here we used that $D^lu_j\overset{\ast}{\rightharpoonup}D^lu_0$ in $L^{\infty}(\Omega)$, and hence $\|D^lu_0\|_{L^{\infty}}\leq M$.\\
		Let $\delta>0$. As $\lambda_j\rightarrow 0$, by using the result of the previous paragraph corresponding to Proposition \ref{prop:muellerprop7undcor8} there exists $j_0(U,V,\varphi,\delta)$ such that for all $j\geq j_0$ there is a function $g_j\in W^{l,\infty}(\Omega,\R^m)$ with the properties
		\begin{align*}
			\{w_j\neq g_j\}&\subset U,\\
			|\{w_j\neq g_j\}|&<\delta,\\
			\mathcal{B}g_j&\in K_{\delta}\text{ a.e. in }U,\\
			\int\limits_{\Omega}^{}F(\mathcal{B}g_j)\dx&\leq e^{\delta}\int\limits_{\Omega}^{}F(\mathcal{B}w_j)\dx+\delta.
		\end{align*}
		Thus, in particular we obtain
		\begin{align*}
			\int\limits_{\Omega}^{}F(\mathcal{B}g_j)\dx\leq e^{\delta}\int\limits_{\Omega}^{}F(\mathcal{B}u_j)\dx+e^{\delta}L_M^FC_{\mathcal{B}}M(|\Omega\backslash V|+|\{\varphi\neq 1\}\cap V|)+e^{\delta}r_j+\delta.
		\end{align*}
		Now choose $\delta_k:=\frac{1}{k}$ and $V_k\subset \subset \tilde{U}_k\subset \subset \Omega$ as well as $\varphi\in C_c^{\infty}(V_k)$ such that $|\tilde{U}_k\backslash V_k|<\frac{1}{k}$, $|\{\varphi_k\neq 1\}\cap V_k|<\frac{1}{k}$, and $0\leq \varphi_k\leq 1$ as in Step 3 in the proof of Theorem 4 in \cite{M99}. Since $\Omega$ is bounded, we can furthermore assume $|\Omega\backslash\tilde{U}_k|<\frac{1}{k}$. Then using the previous step of the proof we infer that there exist $j_k$ such that for all $j\geq j_k$ there exist $g_j$ such that
		\begin{align}
			 g_j&= u_0\text{ on }\Omega\backslash \tilde{U}_k,\\
			 |\{g_j\neq u_j\}|&<\frac{4}{k},\nonumber\\
			 \dist(\mathcal{B}g_j,K)&\leq\frac{1}{k},\nonumber\\
			 \int\limits_{\Omega}^{}F(\mathcal{B}g_j)\dx&\leq e^{\frac{1}{k}}\left(\int\limits_{\Omega}^{}F(\mathcal{B}u_j)\dx+L_M^FC_{\mathcal{B}}M\frac{3}{k}+r_j \right)+\frac{1}{k}.\label{eq:integralboundfinal}
		\end{align}
		Thus, choosing $U_j:=\tilde{U}_k$ if $j_k\leq j\leq j_{k+1}$ yields additionally to the results of Theorem \ref{theo:mueller4} that
		\begin{align*}
			\underset{j\rightarrow \infty}{\limsup}\int\limits_{\Omega}^{}F(\mathcal{B}g_j)\dx\leq \underset{j\rightarrow\infty}{\limsup}\int\limits_{\Omega}^{}F(\mathcal{B}u_j)\dx,
		\end{align*}
		which finishes the proof.
	\end{proof}
	\begin{rem}\label{rem:questiontwo}
		In the special case $\mathcal{B}=D^2$ we can in fact discard the $L^{\infty}$-bound and essentially combining the interpolation method from the proof of Lemma \ref{lemm:muellerlemma5} with an adapted version of M\"uller's original proof yields the result of Theorems \ref{theo:muellertheo2} and \ref{theo:mueller4} for this particular case.\\
		Now let $m\leq N\cdot d$ and let $\mathcal{B}\colon C^{\infty}(\R^d,\R^m)\to C^{\infty}(\R^d,\R^k)$ be a general homogeneous differential operator of first order. There exists a linear map $\mathbb{B}\colon\R^{N\cdot d}\to \R^m$ such that $\mathcal{B}u=\mathbb{B}\cdot Du$ for all $u\in C^{\infty}(\R^d,\R^m)$. Then observe that
		\begin{align*}
			\mathbb{B}\times \operatorname{pr}_{\ker\mathbb{B}}\colon \R^{N\cdot d}\to \image\mathbb{B}\times \ker\mathbb{B}
		\end{align*}
		is surjective and hence bijective as $\image\mathbb{B}\times \ker\mathbb{B}\simeq \R^{N\cdot d}$. Let us denote by $T_{\mathbb{B}}$ the inverse of the above mapping. In particular, for smooth enough $u$ it holds that
		\begin{align*}
			T_{\mathbb{B}}\cdot \left(\mathcal{B}u, \operatorname{pr}_{\ker\mathbb{B}} Du\right)=Du\ \ \ \text{ and }\ \ \ \mathbb{B}\cdot T_{\mathbb{B}}=\operatorname{pr}^{\image\mathbb{B}\times \ker\mathbb{B}}_{\image\mathbb{B}}.
		\end{align*}
		Let now $(u_j)\subset W^{1,\infty}(\Omega)$ be such that $\|Du_j\|_{L^{\infty}}\leq M|K|_{\infty}$ and $\dist(\mathcal{B}u_j,K)\overset{L^1}{\rightarrow}0$. Then
		\begin{align*}
			\dist\left(Du_j,T_{\mathbb{B}}\left(\operatorname{pr}^{\R^m}_{\image\mathbb{B}}K\times B_M^{\operatorname{dim}(\ker\mathbb{B})}(0)\right)\right)\leq \|T_{\mathbb{B}}\|\dist\left(\mathcal{B}u_j,\operatorname{pr}^{\R^m}_{\image\mathbb{B}}K\right)\overset{L^1}{\rightarrow}0,
		\end{align*}
		since $\dist\left(\operatorname{pr}_{\ker\mathbb{B}}Du_j,B_M^{\operatorname{dim}(\ker\mathbb{B})}(0)\right)=0$ and $\dist\left(\mathcal{B}u_j,\operatorname{pr}^{\R^m}_{\image\mathbb{B}}K\right)= \dist(\mathcal{B}u_j,K)$. M\"uller's results for the gradient then provide us with a sequence $(w_j)$ such that
		\begin{align*}
			\dist\left(Dw_j,T_{\mathbb{B}}\left(\operatorname{pr}_{\image\mathbb{B}}K\times B_M^{\operatorname{dim}(\ker\mathbb{B})}(0) \right)\right)\overset{L^{\infty}}{\rightarrow}0.
		\end{align*}
		Hence,
		\begin{align*}
			\dist\left(\mathcal{B}w_j,K \right)&=\dist\left(\mathcal{B}w_j,\operatorname{pr}^{\R^m}_{\image\mathbb{B}}K \right)=\dist\left(\mathbb{B}\cdot Dw_j,\mathbb{B}\cdot T_{\mathbb{B}}\left(\operatorname{pr}^{\R^m}_{\image\mathbb{B}}K\times B_M^{\operatorname{dim}(\ker\mathbb{B})}(0) \right) \right)\\
			&\leq \|\mathbb{B}\|\dist\left(Dw_j,T_{\mathbb{B}}\left(\operatorname{pr}^{\R^m}_{\image\mathbb{B}}K\times B_M^{\operatorname{dim}(\ker\mathbb{B})}(0) \right)\right)\overset{L^{\infty}}{\rightarrow}0.
		\end{align*}
		We thus could derive the results of Theorems \ref{theo:muellertheo2} and \ref{theo:mueller4} for general first order operators $\mathcal{B}$ from those for the gradient. The same procedure can be carried out for second order operators $\tilde{\mathcal{B}}$ by writing $\tilde{\mathcal{B}}u=\mathbb{B}\cdot D^2u$ for every smooth $u$ and some fixed matrix $\mathbb{B}$.\\
		Note that for the method presented in this remark, we still need to assume the uniform $\|\cdot\|_{L^{\infty}}$-bound to infer the result for general differential operators. More importantly, our strategy of proving all the auxiliary results for $\mathcal{B}u$ directly keeping track of the $\|\cdot\|_{L^{\infty}}$-norms offered us access to showing further properties of our truncation method like the conservation of convex integral bounds as in Corollary \ref{cor:convexintegralbound}. For that an important ingredient was using that $F$ is locally Lipschitz in combination with the uniform boundedness of the functions involved.
	\end{rem}
	\section{Two examples of linear constraints with potentials of first or second order}\label{sect:twoexamples}
	The linear constraint $\curl f=0$ with potential $f=\nabla u$ can be handled by the well-known techniques of M\"uller \cite{M99} and Zhang \cite{Z92}. In the following, we want to present two examples of homogeneous differential operators which may not be treated by the latter frameworks. However, these operators have potentials of order one or two such that our results can be applied.
	\subsection{Symmetric gradient}
	The symmetric gradient of a vector field $u\colon\Omega\to \R^d$
	\begin{align*}
		e(u):=\frac{1}{2}(\nabla u+\nabla^{\operatorname{T}}u)
	\end{align*}
	is a first order homogeneous differential operator that appears for example in the theory of linear elasticity. One cannot estimate the gradient by this operator neither in the $L^1$-norm nor in the $L^{\infty}$-norm. Hence, the results regarding the gradient are not applicable for this case. The corresponding linear constraint for which the symmetric gradient is the potential reads as follows
	\begin{align*}
		\tilde{A}f:=\left(\sum\limits_{i=1}^{d}\frac{\partial^2f_{ij}}{\partial x_i\partial_k}+\frac{\partial^2f_{ik}}{\partial x_i\partial_j}-\frac{\partial^2f_{ii}}{\partial x_j\partial_k}-\frac{\partial^2f_{jk}}{\partial x_i\partial_i}\right)_{1\leq j,k\leq d}=0,
	\end{align*}
	see example 3.10 in \cite{FM99}.
	\subsection{A potential for the linearized isentropic Euler equations}\label{sect:euler}
	Now, we will start with a certain linear constraint and construct the corresponding potential. This potential will turn out to be of second order.\\
	Consider the linearization of the isentropic Euler system on $\R^{d+1}$
	\begin{align}
		\begin{split}
		\partial_t m+\diverg M+\nabla q&=0,\\
		\partial_t\rho+\diverg m&=0\label{eq:CEsubsolution}
		\end{split}
	\end{align}
	for the unknowns $z=(\rho,m,M,q)\in \R^+\times \R^d\times S_0^d\times\R^+\subset \R\times \R^d\times S_0^d\times\R\simeq \R^N$ with $N=\left(1+\frac{d}{2}\right)(d+1)$. In the following let $\mathcal{A}\colon C^{\infty}(\R^{d+1},\R^N)\to C^{\infty}(\R^{d+1},\R^d)$ be the homogeneous first order differential operator implementing (\ref{eq:CEsubsolution}), i.e. (\ref{eq:CEsubsolution}) is equivalent to $\mathcal{A}z=0$.\\
	There is an ongoing joint project with Emil Wiedemann in which the uniform convergence to a certain set for generating sequences of Young measures is crucial. The considered generating sequences in this project are $\mathcal{A}$-free with $\mathcal{A}$ as above. These bounds are provided by Corollary \ref{cor:muellercorollary3} if we find a homogeneous potential operator $\mathcal{B}$ of order at most two. For the case $d=3$ Emil Wiedemann showed me in a private communication how to derive such a potential. His proof essentially generalizes to every $d\geq 1$ which we now demonstrate. Set $\tilde{N}:=\frac{1}{4}(d+1)^2\cdot d^2$.
	\begin{prop}\label{prop:potentialconstructioneuler}
		Let $\mathcal{A}$ be as above. There exists a linear homogeneous partial differential operator $\mathcal{B}\colon C^{\infty}\left(\R^{d+1},\R^{\tilde{N}}\right)\to C^{\infty}(\R^{d+1},\R^N)$ of order two such that $\ker\mathcal{A}=\operatorname{im}\mathcal{B}$.
	\end{prop}
	\begin{proof}
		Let $z=(\rho,m,M,q)\in C^{\infty}(\R^{d+1},\R\times \R^d\times S_0^d\times\R)$ be such that $\mathcal{A}z=0$. Note that the map
		\begin{align*}
			\R\times \R^d\times S_0^d\times\R\to S^{d+1},\ (\rho,m,M,q)\mapsto U:=\begin{pmatrix}
			M+qI_d& m\\
			m& \rho
			\end{pmatrix}
		\end{align*}
		is an isomorphism between vector spaces since $M$ is trace-free. Using this identification the PDE (\ref{eq:CEsubsolution}) is equivalent to
		\begin{align*}
			\operatorname{div}_{(t,x)}U(t,x)=0.
		\end{align*}
		From now on the divergence will be understood with respect to $(t,x)$ and hence we simply write $\operatorname{div}$ for $\operatorname{div}_{(t,x)}$.\\
		Every row of $U$ is divergence-free. Thus, Poincar\'e's lemma provides us for every $i=1,...,d+1$ with an antisymmetric matrix-field $\varphi^i$ such that for all $j=1,...,d+1$
		\begin{align}
			(\operatorname{div}\varphi^i)_j=U_{ij}.\label{eq:poincare1}
		\end{align}
		Since $U$ is symmetric, we infer for all $i,j=1,...,d+1$
		\begin{align*}
			\sum\limits_{k=1}^{d}\partial_k(\varphi^i_{jk}-\varphi^j_{ik})=0.
		\end{align*}
		So, we may apply Poincar\'e's lemma once more to obtain antisymmetric matrix fields $\psi^{ij}$ such that
		\begin{align}
			(\operatorname{div}\psi^{ij})_k=\varphi^i_{jk}-\varphi^j_{ik}.\label{eq:poincare2}
		\end{align}
		\textbf{Claim: }The $\tilde{N}=\frac{1}{4}(d+1)^2\cdot d^2$ out of $d^4$ functions $\psi^{ij}_{kl}$ corresponding to the choices $1\leq i<j\leq d+1$ and $1\leq k<l \leq d+1$ suffice to determine $U$.\\
		Indeed, setting $k=i$ or $k=j$ in (\ref{eq:poincare2}) and using the antisymmetry of $\varphi$ yields the functions $\varphi^i_{ij}$ for $i,j=1,...,d+1$. Now fix pairwise distinct $i,j,k$. As $\varphi$ has three indices, there are six functions that have to be determined. For every triple $i,j,k$ the corresponding six equations (\ref{eq:poincare2}) decouple into two systems of three equations due to the antisymmetry of $\varphi^i,\varphi^j,\varphi^k$. Without loss of generality assume $i<j<k$. One of these systems then reads
		\begin{align*}
			(\operatorname{div}\psi^{ij})_k&=\varphi^i_{jk}-\varphi^j_{ik},\\
			(\operatorname{div}\psi^{ik})_j&=-\varphi^i_{jk}-\varphi^k_{ij},\\
			(\operatorname{div}\psi^{jk})_i&=-\varphi^j_{ik}+\varphi^k_{ij},
		\end{align*}
		where only those functions $\psi$ with left superscript larger than the right superscript were needed. Observe that
		\begin{align}
			\begin{pmatrix}
			\varphi^i_{jk}-\varphi^j_{ik}\\
			-\varphi^i_{jk}-\varphi^k_{ij}\\
			-\varphi^j_{ik}+\varphi^k_{ij}
			\end{pmatrix}=\begin{pmatrix}
			1& -1& 0\\
			-1& 0& -1\\
			0& -1& 1
			\end{pmatrix}\cdot \begin{pmatrix}
			\varphi^i_{jk}\\
			\varphi^j_{ik}\\
			\varphi^k_{ij}
			\end{pmatrix}.\label{eq:invertible}
		\end{align}
		The matrix in the above equation is invertible. Thus, the considered $\psi$ uniquely determine $\varphi^i_{jk},\varphi^j_{ik},\varphi^k_{ij}$ and hence by antisymmetry also the other three $\varphi$. In this way we obtain all $\varphi$, which by (\ref{eq:poincare1}) yields $U$. This proves the claim.\\
		By inverting the transformation (\ref{eq:invertible}) the just described procedure can be summarized into a linear homogeneous differential operator $\mathcal{B}$ of order two converting the $\tilde{N}$ functions $\psi^{ij}_{kl}$ into the matrix-field $U$. Thus, $\ker\mathcal{A}\subset \operatorname{im}\mathcal{B}$.\\
		Conversely, for any choice of functions $\left(\psi^{ij}_{kl} \right)^{i<j}_{k<l}$ the implementation $\mathcal{B}\psi=:U$ yields a divergence-free $U$ by basically reversing the previous arguments and observing that the divergence of an antisymmetric smooth matrix-field yields a divergence-free vector field. 
	\end{proof}
	We want to conclude the paper with a remark about the applicability of our results in view of the assumption $\|D^lu\|_{L^{\infty}}\leq M|K|_{\infty}$.
	\begin{rem}\label{rem:refereequeryone}
		The proof of the results in this paper relied heavily on the assumption $\|D^lu_j\|_{L^{\infty}}\leq M|K|_{\infty}$, where $u_j$ is the considered sequence we want to truncate. This assumption has the technical reason to provide us with pointwise bounds used in the proofs. But this kind of uniform $L^{\infty}$-bound occurs naturally for example in sequences of solutions to the linearized compressible Euler equations with one space dimension:\\
		Let $z_j$ be a sequence of solutions of (\ref{eq:CEsubsolution}) for $(t,x)\in(0,1)\times (0,1)$, i.e. $d=1$, with the property that $\|z_j\|_{L^{\infty}}\leq M$ for some $M>0$. Such uniformly bounded solutions $z$ can be obtained from bounded solutions $(\rho,u)$ of the compressible Euler equations by setting $z:=\left(\rho,u,0,\frac{u^2}{\rho}\right)$ if we require $\rho$ to be uniformly bounded away from zero. Further, one can show that solutions $(\rho,u)$ of this type arise from sufficiently regular initial data $\left(\rho^0,u^0\right)$ satisfying that $\rho^0$ is uniformly bounded away from zero, cf. section 4 in \cite{D83}.\\
		Now the construction in Proposition \ref{prop:potentialconstructioneuler} shows that the potential of $z$ is obtained by applying Poincar\'e's lemma twice and via an invertible linear map. Observe that if $v$ is a two-dimensional divergence-free vector field in two dimensions, then $v=\operatorname{div} \varphi$, where $\varphi=\begin{pmatrix}
			0 & g\\-g & 0
		\end{pmatrix}$ for some function $g$. In particular, $\partial_tg=v_1$ and $\partial_xg=v_2$. Hence, $\|Dg\|_{L^{\infty}}\leq \|v\|_{L^{\infty}}$. Now if we choose $g$ such that $\dashint g\dt\dx =0$, then by the Poincar\'e-Wirtinger Inequality we get that $\|g\|_{L^{\infty}}\leq c\|v\|_{L^{\infty}}$. We repeat this argument to infer that $\|D^2u\|_{L^{\infty}((0,1)\times (0,1))}\leq C\|z\|_{L^{\infty}((0,1)\times (0,1))}$ for some fixed constant $C>0$, where $u$ is the potential of $z$, i.e. $\mathcal{B}u=z$.
	\end{rem}
	\section*{Acknowledgements}
	The author wants to thank Emil Wiedemann for proposing the problem and for very interesting and fruitful conversations. The author is also very grateful for the interesting suggestions of the unknown referee from Calculus of Variations and Partial Differential Equations, who proposed the method presented in Remark \ref{rem:questiontwo} and who inspired Remark \ref{rem:refereequeryone}.

\end{document}